\documentclass[11pt, a4paper]{amsart}
\usepackage{latexsym,mathrsfs,bm}
\usepackage{graphicx,color,url}
\usepackage{amssymb,mathrsfs}
\usepackage{bm}
\usepackage{hyperref}

\usepackage[english]{babel}
\usepackage{paralist}
\usepackage{amsthm,amsfonts,amssymb,amsmath}
\usepackage{mathtools}
\usepackage[utf8]{inputenc}

\newcommand{\R}{\mathbb{R}}
\newcommand{\Q}{\mathbb{Q}}
\newcommand{\F}{\mathbb{F}}

\newcommand{\N}{\mathbb{N}}

\newcommand{\Z}{\mathbb{Z}}

\newcommand{\es}[1]{\begin{equation}\begin{split}#1\end{split}\end{equation}}
\newcommand{\est}[1]{\begin{equation*}\begin{split}#1\end{split}\end{equation*}}

\newcommand{\PP}{\mathbb{P}}

\newcommand{\nequiv}{\not\equiv}

\newcommand{\tn}[1]{\textnormal{#1}}

\newcommand{\leg}[2]{\left(\frac{#1}{#2}\right)}

\newtheorem*{theo*}{Theorem}
\newtheorem{theo}{Theorem}

\newtheorem{ezer}{Exercise}

\newtheorem{conj}[ezer]{Conjecture}

\newtheorem{lemma}{Lemma}

\newtheorem{remark}{Remark}

\newtheorem*{rem*}{Remark}
\def\sumstar{\operatornamewithlimits{\sum\nolimits^*}}

\newcommand{\Res}{\operatorname*{Res}}

\let\originalleft\left
\let\originalright\right
\renewcommand{\left}{\mathopen{}\mathclose\bgroup\originalleft}
\renewcommand{\right}{\aftergroup\egroup\originalright}

\definecolor{pink}{rgb}{1,.2,.6}
\definecolor{orange}{rgb}{0.7,0.3,0}
\definecolor{blue}{rgb}{.2,.6,.75}
\definecolor{green}{rgb}{.4,.7,.4}

\newcommand{\suppress}[1]{}

\newcommand{\Es}{{\mathcal E}}
\newcommand{\Sf}{{\mathcal S}}
\newcommand{\Rf}{{\mathcal R}}
\newcommand{\rf}{r_{\Es}}

\numberwithin{equation}{section}

\title{On the typical rank of elliptic curves over $\Q(T)$} 

\author[F.~Battistoni]{Francesco Battistoni}
\address{Laboratoire de math\'ematiques de Besan\c con, Universit\'e de Bourgogne Franche-Comt\'e, CNRS UMR 6623, 16, route de ray, 25030 Besan\c con cedex, France}
\email{francesco.battistoni@univ-fcomte.fr}

\author[S.~Bettin]{Sandro Bettin}
\address{Dipartimento di Matematica, Universit\`a di Genova, Via Dodecaneso 35, 16146 Genova, Italy}
\email{bettin@dima.unige.it}

\author[C.~Delaunay]{Christophe Delaunay}
\address{Laboratoire de math\'ematiques de Besan\c con, Universit\'e de Bourgogne Franche-Comt\'e, CNRS UMR 6623, 16, route de ray, 25030 Besan\c con cedex, France}
\email{christophe.delaunay@univ-fcomte.fr}

\makeatletter
\@namedef{subjclassname@2020}{%
  \textup{2020} Mathematics Subject Classification}
\makeatother

\begin{document}

\begin{abstract}
We give upper bounds for the number of elliptic curves defined over $\Q(T)$ in some families having positive rank, obtaining in particular that these form a subset of density zero. This confirms Cowan's conjecture~\cite{Cowan} in the case $m,n\leq2$.

\end{abstract}



\maketitle

\section{Introduction}
Let  ${\Es}$ be an elliptic curve defined over $\Q(T)$ given by a Weierstrass equation
\begin{equation}\label{Weq}
\Es \colon \quad Y^2 = X^3 + \alpha_2(T)X^2 + \alpha_4(T)X+\alpha_6(T),
\end{equation}
where $\alpha_2(T), \alpha_4(T)$ and $\alpha_6(T)$ are polynomials in, say, $\Z[T]$. Since $\Es$ is an elliptic curve, it has to be non-singular meaning that the discriminant 
of equation~\eqref{Weq} is not $0$ as a polynomial. Furthermore, we say that $\Es$ is iso-trivial if the $j$-invariant associated to \eqref{Weq} is a constant rational function 
in $\Q(T)$. The group of $\Q(T)$-rational points of $\Es$ is finitely generated and we denote by $\rf$ the rank of $\Es(\Q(T))$.
The rank $\rf$ is an important arithmetical invariant 
of elliptic curves over $\Q(T)$ and is related to several questions in number theory. For example, Silverman's specialization theorem asserts that for almost all 
specializations of $T$ at $t \in \Q$, the rank of the associated elliptic curve defined over $\Q$ is at least $\rf$. This has direct consequences on
the study of the distribution of ranks of families of elliptic curves defined over $\Q$, or over number fields $K$ \cite{miller, david_all, stewart_top, rubin_silverberg} and on the research of high rank elliptic curves \cite{mestre, fermigier, arms_all}.
The study of the rank, $\rf$, has also 
some impact on a question in arithmetic geometry asking whether the set of $\Q$-rational points of equation~\eqref{Weq} seen as an hypersurface in the three 
dimensional affine  space over $\Q$ is Zariski-dense.
The question is positively answered whenever $\rf >0$, and if  $\rf=0$ a sufficient criterion consists in showing that there 
exist  infinitely many specializations of  $T$ at $t \in \Q$ such that the rank of the corresponding curve over $\Q$ is positive (see \cite{mazur} and recent works by J.~Desjardins on the 
subject, \cite{desjardins1, desjardins2}). Under the parity conjecture this can be done, for example, by studying the behavior of the root 
numbers of the specializations. 

In a recent work~\cite{Cowan}, Cowan considered the typical value of $\rf$, giving an heuristic argument suggesting that for most natural families of elliptic curves over 
$\Q(T)$ the rank $\rf$ is zero for almost all member of the family. More specifically, letting $\| P \|$ be the absolute height\footnote{Cowan stated the conjecture for the Mahler measure $\mu$, mentioning that one could also choose the height instead. By the inequalities $\binom{n}{[n/2]}^{-1}\|P\|\leq \mu(P)\leq \sqrt{n+1}\|P\|$ for $P$ of degree $n$ (see \cite[Lemma 1.6.7]{Bombieri_Gubler}) the two choices are equivalent.} 
 of a polynomial $P\in\Z[X]$ (i.e. the maximum of the absolute values of the coefficients of $P$), and writing for any $m,n,H\in \N$ 
\est{
\Sf_{m,n}&:=\{\Es \mid \alpha_2=0,\ \deg(\alpha_4)\leq m,\ \deg(\alpha_6)\leq n,\ \Es \text{ non-singular} \},\\
\Sf_{m,n}(H)&:=\{\Es\in \mathcal F_{m,n} \mid \| \alpha_4\|<H^3,\ \|\alpha_6\|<H^2 \},
}
he made the following conjecture.
\begin{conj}\label{cowan}
Let $m,n\in\N$. Then, as $H\to\infty$ we have
\es{\label{avr}
\lim_{H\to\infty}\frac{1}{\sharp\Sf_{m,n}(H)}\sum_{\Es\in\Sf_{m,n}(H)}\rf=0.
}
In particular, 
\es{\label{tyr}
\lim_{H\to\infty}\frac{\sharp\{\Es \in\Sf_{m,n}(H)\mid \rf\geq1 \}}{\sharp\Sf_{m,n}(H)}=0,
}
that is almost all elliptic curves defined over $\Q(T)$ have rank zero.
\end{conj}

Notice that this conjecture is in contrast with what is expected to hold for elliptic curves over $\Q$. Indeed, considerations on the root numbers lead to the belief that for most 
natural families of elliptic curves over $\Q$ one has that $50\%$ of the curves have rank $1$ and $50\%$ have rank 0. See for example~\cite{goldfeld} for the case of 
quadratic twists and see~\cite{BDD} for the study of some exceptional families.

\medskip

In this note we use the results of \cite{BBD} to prove Cowan's conjecture in the case $m,n\leq2$. Notice that in this case the elliptic surfaces defined by \eqref{Weq} over 
$\PP^1_\Q$ are all rational elliptic surfaces since $\deg(\alpha_i)\leq i$ and, with $m,n\leq2$, the discriminant of $\Es$ can not be a constant time the twelfth power 
of a degree one polynomial (cf.~\cite{schutt_shioda}). Rational elliptic surfaces over $\PP^1_\Q$ are the ones with a geometric genus, $g_{\Es}$, equal to zero leading to 
the simplest geometry  and an easier control of the Mordell-Weil rank over $\Q(T)$ of the associated elliptic curve. More precisely, we obtain the following result.

\begin{theo}\label{mainc} 
Conjecture~\ref{cowan} holds for $1\leq m,n\leq2$.
\end{theo}

We also consider the problem for other families of elliptic curves over $\Q(T)$, where $\alpha_2$ is not required to be zero. In the case of $\deg(\alpha_i)\leq2$ for $i=2,4,6$, we can rewrite the right hand side of~\eqref{Weq} as a polynomial in $T$, writing $\Es=\Es_{A,B,C}$ as
\es{\label{eint}
\Es_{A,B,C} \colon \quad Y^2 = A(X)T^2 +B(X)T+C(X)
}
where $A(X), B(X)$ and $C(X)-X^3$ are polynomials in $\Z[X]$ of degree $\leq2$. For $P_1,\dots,P_n\in\Z[X]$ we let $\|P_1,\dots,P_n\|:=\max(\|P_1\|,\dots,\|P_n\|)$.
We shall consider the following 3 families, the last one being different from $\mathcal S_{2,2}(H)$ only for the different ordering.\footnote{It would be possible to use alternative orderings also for these families.}  
\est{
\Sf_{0}(H)&:=\{\Es_{A,B,C}\mid A=0,\ \deg(B), \deg(C-X^3)\leq2,\ \|B,C\|\leq H,\, \Es \text{ non-singular}\},\\
\Sf_{\square}(H)&:=\{\Es_{A^2,B,C}\mid \deg(A)=0,\ \deg(B),\deg(C-X^3)\leq2,\ \|A,B,C\|\leq H,\, \Es \text{ non-singular}\},\\
\Sf_{\ell}(H)&:=\{\Es_{A,B,C}\mid \deg(A), \deg(B),\deg(C-X^3)\leq1,\ \|A,B,C\|\leq H,\, \Es \text{ non-singular}\}.
}
Notice that the case $\Sf_0$ corresponds to elliptic curves defined by $\eqref{Weq}$ 
with $\deg( \alpha_2) \leq 1$, $\deg (\alpha_4) \leq 1$ and 
$\deg (\alpha_6) \leq 1$. The case $\Sf_\square$ corresponds to $\eqref{Weq}$ with $\deg (\alpha_2) \leq 1$, $\deg (\alpha_4) \leq 1$
 and $\alpha_6$ 
of degree 2 with a square leading coefficient. Finally, the case $\Sf_\ell$ corresponds to $\eqref{Weq}$ with 
$\alpha_2(T)=0$ and $\deg \alpha_4(T) \leq 2$, $\deg \alpha_6(T) \leq 2$. The elliptic curves in all of these families have rank bounded by 5 over $\Q(T)$  (cf.~\cite{BBD}) 
whenever the elliptic curve is non-isotrivial. In our case ($\deg A, \deg B \leq 2$), it is not difficult to see that $\Es_{A,B,C}$ is iso-trivial if and only if $A=B=0$ (this is not 
true in general for elliptic curves over $\Q(T)$). In particular, the contribution of isotrivial cases can be easily bounded by~\cite{bhargava_shankar} 
(see Section~\ref{pmta}) and we have that the analogous of~\eqref{avr} and~\eqref{tyr} are equivalent in our case.

Also for these families we can show that almost all elliptic curves have rank zero over $\Q(T)$.
\begin{theo}\label{maina} 
For any $*\in\{0,\square,\ell\}$ let
$\Rf_*(H):=\{\Es_{A,B,C}\in \Sf_*(H)\mid r_{\Es_{A,B,C}}\geq1\}$.
Then,
\es{\label{mainae}
\lim_{H\to \infty }\frac{\sharp {\mathcal R}_*(H)}{\sharp \Sf_*(H)} =0.
}
\end{theo}
\begin{remark}\label{rer}
Our method is based on a simple application of the Tur\'an sieve and actually gives quantitative bounds. In particular for Theorem~\ref{maina} we show that 
$$\frac{\sharp {\mathcal R}_*(H)}{\sharp \Sf_*(H)}=O(H^{-1/3} (\log H \log \log H)^{2/3}).$$
It is very likely that one could improve this bound by a more sophisticated method.
Our proof gives an upper bound for the number of reducible polynomials in certain sets of polynomials of low 
degree. These polynomials should behave as generic polynomials of the same degree and thus one would expect that $\frac{\sharp {\mathcal R}_*(H)}{\sharp \Sf_*(H)}=O(H^{-1+\varepsilon})$ (see \cite{kuba}). Some numerical experiments suggest that this should indeed be the case.
\end{remark}
\begin{remark}\label{rer2}
One of the key ingredient to apply Tur\'an sieve is Lemma~\ref{br}, which is not expected to generalize to the case of $\alpha_2$, $\alpha_4$ and $\alpha_6$ with degree large enough, in particular whenever the underlying elliptic surface is a $K3$ surface. However, it seems natural to expect 
a zero density result in this case too, as explained by Cowan (\cite{Cowan}).
\end{remark}

Conjecture~\ref{cowan} and Theorems~\ref{mainc} and~\ref{maina} suggest that the rank over $\Q(T)$ of elliptic curves is typically expected to be zero. Thus, the study of
 the Zariski density of $\Q$-rational points of \eqref{Weq} seen in the affine three space cannot, in general, be addressed by rank considerations only. However, in general the root number should be, under the Chowla and  squarefree conjectures, equidistributed and hence there should exist 
infinitely many specializations with rank $\geq 1$ over $\Q$, implying (conditionally) the Zariski density (see \cite{helfgott,desjardins}). This 
strategy does not apply whenever the root number is not equidistributed: this happens to be the case for potentially parity biased families as studied 
in \cite{BDD}, where it is proved that there are essentially 6 different classes of such non-isotrivial families of the form \eqref{Weq} with 
$ \deg(\alpha_2),\deg( \alpha_4),\deg( \alpha_6) \leq 2$. Among those 6 classes, two correspond to the families considered in this paper. These are 
\begin{align*}
{\Es}& \colon \quad wY^2 = X^3 + 3TX^2 + 3sX+sT \quad (*=0) \\
{\mathcal G} & \colon \quad  Y^2 = X^3 +3TX^2 + 3TX + T^2 \quad (*=\square) 
\end{align*}
It is proved unconditionally in \cite{desjardins2} that the family ${\mathcal G}$ has infinitely many integer specializations with negative root number. In \cite{desjardins3}, the authors investigate, in particular, the situation given by (sub)-families of ${\Es}$ 
and describe the possibilities to have families with constant root number.  
\medskip
~\\
The organization of the paper is the following. In Section~\ref{prel} we give some lemmas which we will need in the proofs. Among these, we mention in particular Lemma~\ref{br} (from~\cite{BBD}) which shows that it suffices to bound the number of reducible polynomials in certain sets. In Section~\ref{pmta}, we prove Theorem~\ref{maina} by showing how to obtain such bounds as a consequence of the Tur\'an sieve. In Section~\ref{pmtc} we show how the same proof can be adapted (and in fact simplified) to prove Theorem~\ref{mainc}.

\subsection*{Acknowledgments} The first and the last authors are supported  by the French ``Investissements d'Avenir" program, project ISITE-BFC (contract ANR-lS-IDEX-OOOB). They are also members of the Laboratoire de math\'ematiques de Besan\c con that receives support from the EIPHI Graduate School (contract ANR-17-EURE-0002). The second author is member of the INdAM group GNAMPA and his work is partially supported by PRIN 2017 ``Geometric, algebraic and analytic methods in arithmetic''.

\section{Preliminaries}\label{prel}
We start with the following standard application of the Tur\'an sieve. We include a proof for the sake of completeness.

\begin{lemma}\label{lt}
Let $f(a_1,\cdots,a_n,X) \in \Z[a_1,\cdots,a_n][X]$ with $n>1$. Let ${\mathcal P} $ be a set of primes such that $\mathcal P\cap[1,x]\sim C \frac{x}{\log x}$ as $x\to\infty$ for some $C>0$. Assume there exists $\delta>0$ such that for all $p\in\mathcal P$ the set
\est{
{\mathcal A}_p & = \{(a_1,\cdots,a_n) \in \F_p^n \mid f(a_1,\cdots,a_n,X) \mbox{ is irreducible in } \F_p[X] \}
}
satisfies $\sharp {\mathcal A}_p = \frac{p^n}{\delta}  + O(p^{n-1})$ as $p\to\infty$. Finally, for $ \boldsymbol H =(H_1,\dots,H_n)\in\R_{>0}^n$, let 
$$
{\mathcal B}(\boldsymbol H) :=\{(a_1,\cdots,a_n) \in \Z^n \mid |a_i| \leq H_i\  \forall i \mbox{ and }   f(a_1,\cdots,a_n,X) \mbox{ is reducible in } \Q[X] \}.
$$
Then, letting $H:=\min(H_1,\dots, H_n)$, as $H\to\infty$ we have
$$
\sharp {\mathcal B}(\boldsymbol H) \ll H_1\cdots H_n \cdot  H^{-1/3} (\log H \log \log H)^{2/3}.
$$
\end{lemma}
\begin{proof}
For all $p \in {\mathcal P}$, we define 
$$
{\mathcal A}_p(\boldsymbol H) = \{(a_1,\cdots,a_n) \in \Z^n \mid |a_i| \leq H_i\ \forall i \mbox{ and }   f(a_1,\cdots,a_n,X) \mbox{ is irreducible in } \F_p[X] \}.
$$
Let $X:=H_1\cdots H_n$. If $p \leq H$, $p\in\mathcal P$, then we have
\begin{eqnarray*}
\sharp {\mathcal A}_p(\boldsymbol H) = \frac{X}{p^n} \left( 1 + O\left(\frac{p}{H}\right)\right) \sharp {\mathcal A}_p =   \frac{X}{\delta} + O\left( \frac{X}{p}\right) + O\left(pX/H  \right)=\frac{X}{\delta} + R_p ,
\end{eqnarray*}
say. Moreover, by the Chinese Remainder Theorem, if $p, q \in {\mathcal P}$ satisfy $p, q \leq H^{1/2}$, then
\begin{align*}
\sharp ({\mathcal A}_p(H) \cap {\mathcal A}_q(H)) &= \frac{X}{(pq)^n} \left( 1 + O\left(\frac{pq}{H}\right)\right) \sharp {\mathcal A}_p  \sharp {\mathcal A}_q \\
&  = \frac{X}{\delta^2} +  O\left(  \frac{X}{p} \right) + O\left(  \frac{X}{q} \right) + O\left( pq X/H\right)= \frac{X}{\delta^2}  + R_{p,q},
\end{align*}
say. 
Now we define 
\begin{align*}
{\mathcal B}_{{\mathcal P},z}(\boldsymbol H) = \{(a_1,\cdots,a_n) \in \Z^n &\mid  |a_i| \leq H_i\ \forall i \mbox{ and } f(a_1,\cdots,a_n,X)  \\  &\quad \mbox{ is reducible in } \F_p[X] \mbox{ for all } p \in {\mathcal P}  \mbox{ with } p\leq z\}
\end{align*}
and we apply Tur\'an's sieve \cite[chapter 4]{cojocaru_murty}, which gives that for $z \leq H^{1/2}$ 
\est{
\sharp {\mathcal B}_{{\mathcal P},z}(\boldsymbol H) &\leq \frac{H^n}{\displaystyle \sum_{\substack{p \in {\mathcal P} \\ p\leq z}} \delta} + \frac{2}{\displaystyle \sum_{\substack{p \in {\mathcal P} \\ p\leq z}} \delta} \sum_{\substack{p \in {\mathcal P} \\ p\leq z}} |R_p| + \frac{1}{\big({\displaystyle \sum_{\substack{p \in {\mathcal P} \\ p\leq z}} \delta} \big)^2} \sum_{\substack{p, q \in {\mathcal P} \\ p, q\leq z}} |R_{p,q}|\\
&\ll X \frac{\log z}{z} + X \frac{\log z \log\log z}{z} + z X /H + z^2 X /H \\
&\ll X H^{-1/3} (\log H \log \log H)^{2/3},
}
upon choosing $z=  (H\log H \log\log H)^{1/3}$. The lemma then follows since $\sharp {\mathcal B}(\boldsymbol H) \leq \sharp {\mathcal B}_{{\mathcal P},z}(\boldsymbol H)$.
\end{proof}

When applying this lemma to our cases, we will need some results on the number of irreducible polynomials in $\F_p[x]$ satisfying certain properties. 
\begin{lemma}\label{irpo}
Let $q$ be a prime power.
The number of irreducible, monic polynomials of degree $n\geq 1$ is 
$
\frac{1}{n}\sum_{d\mid n} \mu(d)q^\frac{n}{d},
$
where $\mu$ denotes the M\"obius function.
\end{lemma}
\begin{proof}
See e.g.~\cite[Corollary of Proposition 2.1]{rosen}.
\end{proof}

\begin{lemma}\label{even_polb}
Let $q$ be an odd prime power. The number of even, irreducible, monic polynomials of degree $n\geq 4$ is
$\frac{1}{n}\sum_{d|n\atop d\,\text{odd}}\mu(d)(q^{\frac{n}{2d}}-1).$
\end{lemma}
\begin{proof}
See~\cite[Theorem~3]{Cohen}.
\end{proof}
\begin{lemma}\label{even_lac_pol}
Let $q$ be an odd prime power, $a\in\F_q$, and $n\geq4$ even. The number of even, irreducible, monic polynomials of degree $n$ such that the coefficient of the $x^{n-2}$ term is $a$ is $\frac1{n}q^{n/2-1}+O(q^{n/4})$.
\end{lemma}
\begin{proof}
By~\cite[Theorem~2]{Cohen} we have that the set of polynomials satisfying the above conditions is in bijection with
\est{
S_a:=\{P\in\F_q[x]\mid P \text{ monic irreducible}, \deg(P)=n', P_{n'-1}=a, (-1)^{n'}P_0\text{ non-square mod }q\},
}
where $n'=n/2$, $P_i$ is the $i$-th coefficient of $P(X)$ and where we used the fact that the condition of $(-1)^{n'}P_0$ (non-zero)  being a non-square modulo $q$ is equivalent to $(2,(q-1)/{e})=1$ with $e$ the order of $(-1)^{n'}P_0$ modulo $q$. The lemma then follows since Theorem~2 in~\cite{carlitz} gives an exact formula for the cardinality of $S_a$, implying in particular that $|S_a|=\frac1{n}q^{n/2-1}+O(q^{n/4})$.
\end{proof}

Finally, we give some bounds for the rank which follow immediately from~\cite{BBD}. We need the following definition
\est{
M_{P_1,P_2}(X) := \Res_Y(P_1(Y),X^2-P_2(Y))
}
for two polynomials $P_1,P_2\in\Z[X]$, with $P_2\neq0$, where the resultant is computed with respect to the variable $Y$. Also, given a polynomial $P\in\Z[X]\setminus\{0\}$, we denote by $\Omega(P)$ the number of irreducible factors of $P$ counted with multiplicity.

\begin{lemma}\label{br}
Let $\Es_{A,B,C}$ as in~\eqref{eint} with $\deg(A),\deg(B)\leq2$ and $C$ monic of degree $3$. Also assume $A$ and $B$ are not both zero. Then $r_{\Es_{A,B,C}}\leq5$. Moreover,
\est{
r_{\Es_{A,B,C}}\leq
\begin{cases}
\Omega(M_{B,C})-1& \text{if $A=0$,}\\
\Omega(B^2-4AC)-1& \text{if $A\in\Z^2_{\neq0}$, or if $A\in\Z$ and $\deg(B)\leq1$}\\
\Omega(M_{B^2-4AC,A})-1& \text{if $\deg(A)=1$.}
\end{cases}
}
\end{lemma}
\begin{proof}
A formula for the rank valid for $\deg(A),\deg(B),\deg(C-X^3)\leq2$ is given in~\cite[Theorem~1]{BBD}. The claimed inequalities then follow easily. We only remark that in the case $A\in\Z\setminus\Z^2$ and $\deg(B)\leq1$,~\cite{BBD} gives that $r_{\Es_{A,B,C}}$ is the number of conjugate classes of roots $\rho$ of $B^2-4AC$ such that $A$ is a square in $\Q(\rho)$. Since $B^2-4AC$ has degree $3$, then such a number is zero unless $B^2-4AC$ has an irreducible factor of degree $2$ and the claimed inequality follows.
\end{proof}

\section{Proof of Theorem~\ref{maina}}\label{pmta}
In this section we prove that~\eqref{mainae} holds for the three families considered. We remark that this implies also that the average rank is zero. Indeed, since for $A,B$ not both zero we have $r\leq5$ by Lemma~\ref{br}, it is sufficient to bound the contribution from $A=B=0$: recall that in our situation, $A=B=0$ correspond exactly to the 
isotrivial case. These iso-trivial families are perhaps most naturally excluded from the definition of the average rank. In any case, one can bound their contribution by~\cite{bhargava_shankar}, since for $H$ sufficiently large we have the very crude bound
\est{
\sumstar_{\deg(C-X^3)\leq2\atop \|C\|\leq H}r_{\Es_{0,0,C}}\leq2 \sumstar_{|c_0|\leq 55 H^3,|c_1|\leq 28H^2}r_{\Es_{0,0,X^3+c_1X+c_0}}\ll H^5
} 
which is sufficient for our purposes, where $\sumstar$ indicates that singular curves are to be excluded from the sum.

\subsection{The case $*=0$}
By Lemma~\ref{br} it suffices to show that almost all polynomials $B,C\in\Z[X]$ with $\deg(B),\deg(C-X^3)\leq2$ are such that $M_{B,C}$ is irreducible. We let $B(X) = b_2X^2 + b_1X + b_0$ and $C(X) = X^3 + c_2X^2+c_1X+c_0$. Then,
\begin{align*}
M_{B,C}(X) &= b_2^3X^4 + \big(-2c_0b_2^3 + (c_1b_1 + 2c_2b_0)b_2^2 + (-c_2b_1^2 - 3b_0b_1)b_2 + b_1^3\big)X^2+{}  \\ 
&\quad+ c_0^2b_2^3 + \big(-c_0c_1b_1 + (-2c_0c_2 + c_1^2)b_0\big)b_2^2 + \big(c_0c_2b_1^2 + (-c_1c_2 + 3c_0)b_0b_1+{} \\ 
&\quad+(c_2^2 - 2c_1)b_0^2\big)b_2 + (-c_0b_1^3 + c_1b_0b_1^2 - c_2b_0^2b_1 + b_0^3).
\end{align*}
Given an irreducible $U(X)=u_4X^4+u_2X^2+u_0\in\F_p[X]$ and a prime $p\geq5$ we need to count the number $\mathcal N_U$ of $(b_2,b_1,b_0,c_2,c_1,c_0)\in\F_p[x]^6$ such that $M_{B,C}(X)\equiv U(X)\mod p,$ i.e. 
\es{\label{sys1}
\begin{cases}
b_2^3 &\equiv  u_4,\\
(-2c_0b_2^3 + (c_1b_1 + 2c_2b_0)b_2^2 + (-c_2b_1^2 - 3b_0b_1)b_2 + b_1^3) &\equiv u_2,\\[0.4em]
\parbox[c]{.7\textwidth}{
$c_0^2b_2^3 + \big(-c_0c_1b_1 + (-2c_0c_2 + c_1^2)b_0\big)b_2^2 + \big(c_0c_2b_1^2 + (-c_1c_2 + 3c_0)b_0b_1+ {}+(c_2^2 - 2c_1)b_0^2\big)b_2 +(-c_0b_1^3 + c_1b_0b_1^2 - c_2b_0^2b_1 + b_0^3).
$}&\equiv u_0.
\end{cases}
}
It is convenient to assume $ p\equiv 2\mod 3$ so that the cubic map is an automorphism of $\F_p$. If $u_4\equiv 0$, then one easily sees that the system has $\mathcal N_U=p^3$ solutions. If $u_4\nequiv 0$, then the first two equations determine unique values for $b_2$ and $c_0$ respectively. Inserting these values in the third equation, this reduces to 
$$\Delta_B b_2^4\bigg(c_1+\frac{3 b_1^2 - 4 b_1 b_2 c_2 + \Delta_B}{4 b_2^2}\bigg)^2\equiv \tilde\Delta_U,$$
where $\Delta_B:=b_1^2-4b_0b_2$ and $\tilde\Delta_U:=u_2^2-4u_0u_4$ and for some (non-zero) $b_2$ determined by $u_4$. Since $\tilde\Delta_U\nequiv 0$, this equation has two solutions in $c_1$ if $\tilde\Delta_U\Delta_B$ is a non-zero square modulo $p$ and no solutions otherwise. Denoting the Legendre symbol by $(\frac \cdot\cdot)$ we obtain
\est{
\mathcal N_U&=\sum_{b_0,b_1,c_2\, \tn{mod } p}\bigg(\bigg(\frac{\Delta_B^2}{p}\bigg)+\bigg(\frac{\tilde\Delta_U\Delta_B}{p}\bigg)\bigg)=p\sum_{b_0,b_1\, \tn{mod } p}\bigg(\bigg(\frac{\Delta_B^2}{p}\bigg)+\bigg(\frac{\tilde\Delta_U\Delta_B}{p}\bigg)\bigg)=p^3-p^2.
}
Thus, letting
${\mathcal A}_p := \{(b_2,b_1,b_0,c_2,c_1,c_0) \in \F_p^6 \mid M_{B,C}(X) \mbox{ is irreducible in } \F_p[X] \}$, we have
\begin{align*}
    \sharp\mathcal{A}_p &= \sum_{\substack{U\in\F_p[X], \deg U\leq 4\\U\text{ even, irreducible}}}\mathcal N_U = \sum_{\substack{U\in\F_p[X], \deg U= 4\\U\text{ even, irreducible}}}p^3+O(p^5),
    = \frac{p^6}{4} + O(p^5)
\end{align*}
by Lemma~\ref{even_polb}. Thus, we obtain the theorem in the case $*=0$ (and error as given in Remark~\ref{rer}) by applying Lemma~\ref{lt} with $\mathcal{P}=\{p\equiv 2\bmod 3\}$ (so that $C=1/2$) and $H_1=\cdots =H_6=H$.

\subsection{The case $*=\square$}
By Lemma~\ref{br} it suffices to show that almost all polynomials $A\in\Z$, $B,C\in\Z[X]$ with $\deg(B),\deg(C-X^3)\leq2$ are such that $B^2-4A^2C$ is irreducible. We let $B$ and $C$ as in the previous subsection and $A=k$. We have
\begin{align*}
B(X)^2-4A^2 C(X) &= b_2^2 X^4 + (2b_1b_2 - 4k^2)X^3 + (2b_0b_2 + (b_1^2 - 4k^2c_2))X^2+{} \\
&\quad+ (2b_0b_1 - 4k^2c_1)X + (b_0^2 - 4k^2c_0).
\end{align*}
Given a prime $p\geq3$ and any $U(X)=u_4X^4+\cdots+ u_0\in\F_p[X]$ irreducible, we need to find the number of solutions $\mathcal N_p$ to $B^2-4A^2C\equiv U,$ i.e. the number of solutions to
\es{\label{sys2}
\begin{cases}
b_2^2 &\equiv  u_4,\\
2b_1b_2 - 4k^2 &\equiv u_3,\\
2b_0b_2 + (b_1^2 - 4k^2c_2)&\equiv u_2,\\
2b_0b_1 - 4k^2c_1 &\equiv u_1,\\
b_0^2 - 4k^2c_0 &\equiv u_0.
\end{cases}
}
First notice that we can assume $k\nequiv0$ since otherwise $B^2-4k^2C$ cannot be irreducible.
If $u_4\equiv 0$, we have $b_2\equiv0$ and we then obtain $(1+(\frac{u_3}{p}))$ choices for $k$. Choosing among the $p^2$ possibilities for the couple $(b_1, b_0)$ the remaining variables $c_2$, $c_1$ and $c_0$ are uniquely determined and thus $\mathcal N_U =(1+(\frac{u_3}{p}))p^2$.

If $u_4\nequiv 0$, we have $(1+(\frac{u_4}{p}))$ choices for $b_2$. Each of these give $(1+(\frac{2b_1b_2-u_3}{p}))$ possible choices for $k$. 
Since $c_2$, $c_1$ and $c_0$ are uniquely determined by the other variables we thus obtain
\est{
\mathcal N_p=\sum_{b_1,b_2\mod p}\bigg(1+\bigg(\frac{u_4}{p}\bigg)\bigg)\bigg(1+\bigg(\frac{2b_1b_2-u_3}{p}\bigg)\bigg)=\bigg(1+\bigg(\frac{u_4}{p}\bigg)\bigg)\bigg(p^2-p+p\bigg(\frac{-u_3}{p}\bigg)\bigg).
}
Thus, for $ {\mathcal A}_p := \{(k,b_2,b_1,b_0,c_2,c_1,c_0) \in \F_p^7 \mid B^2-4A^2C \mbox{ is irreducible in } \F_p[X] \}$ we have
\begin{align*}
   \sharp\mathcal{A}_p &= \sum_{\substack{U\in\F_p[X]\text{ irreducible }\\\deg U\leq 4}}\mathcal N_U = 
   \sum_{\substack{U\in\F_p[X]\text{ irreducible }\\\deg U= 4}}\left(1+\leg{u_4}{p}\right)p+O(p^6)
  \\
   &= p^2\sum_{\substack{U\in\F_p[X]\text{ irreducible }\\\deg U= 4}}1+ O(p^6)   = \frac{p^7}{4} + O(p^6),
\end{align*}
by Lemma~\ref{irpo}. The theorem in this case then follows as before.

\begin{remark}
It would be interesting to have the result also in the case with $A$ in place of $A^2$. In this case the rank is the number of conjugate classes of roots $\rho$ of $B^2-4AC$ such that $A$ is a square in $\Q(\rho)$. In order to show this number is typically zero one would need to prove that the resolvent cubic of $B^2-4kC$ is almost always irreducible.  
\end{remark}

\subsection{The case $*=\ell$}
By Lemma~\ref{br} it suffices to show that almost all polynomials $A,B,C\in\Z[X]$ with $\deg(A),\deg(B),\deg(C-X^3)\leq1$ are such that $M_{B^2-4AC,A}$ is irreducible. We let $A(X)=a_1x+a_0$, $B(X) = b_1X + b_0$ and $C(X) = X^3 +c_1X+c_0$. If $a_1\nequiv0$, which we can assume at a cost of an admissible error, then
\est{
a_1^{-1}M_{B^2-4AC,A}(X)&=   - 4 X^8+ 12 a_0 X^6 + (-12 a_0^2 + a_1 b_1^2 - 
    4 a_1^2 c_1) X^4+ (4 a_0^3 + 2 a_1^2 b_0 b_1 + \\
&\quad- 2 a_0 a_1 b_1^2 - 
    4 a_1^3 c_0 + 4 a_0 a_1^2 c_1) X^2 +(a_1^3 b_0^2 - 2 a_0 a_1^2 b_0 b_1 + 
   a_0^2 a_1 b_1^2).
}
Let $p\geq 5$ be prime and $U(X)=-4X^8+u_6X^6+\cdots +u_2X^2+u_0\in \F_p[X]$ be even and irreducible. We have that $a_1^{-1}M_{B^2-4AC,A}\equiv U$ is equivalent to 
\es{\label{sys3}
\begin{cases}
a_0 & \equiv  u_6/12, \\
-12 a_0^2 + a_1 b_1^2 -   4 a_1^2 c_1& \equiv u_4, \\
4 a_0^3 + 2 a_1^2 b_0 b_1 - 2 a_0 a_1 b_1^2 -     4 a_1^3 c_0 + 4 a_0 a_1^2 c_1& \equiv u_2,\\
a_1^3 b_0^2 - 2 a_0 a_1^2 b_0 b_1 +    a_0^2 a_1 b_1^2& \equiv u_0.
\end{cases}
}
Notice that we can't have $a_1\equiv 0$ since the last equation would imply $u_0\equiv 0$ which is not possible since $U$ is irreducible. We then solve the second equation in $c_1$ inserting the result in the last two equations reducing them to
\est{
\begin{cases}
c_0& \equiv 4(432 a_1^2 b_0 b_1 - 216 u_2 - 18 a_1 b_1^2 u_6 - 18 u_4 u_6 - u_6^3)/(864 a_1^3)\\
 (b_0 - b_1 u_6/(12 a_1))^2& \equiv u_0/a_1^3. 
\end{cases}
}
The second equation implies that $u_0/a_1$ is a square modulo $p$, and so we have $(p-1)/2$ possible choices for $a_1$. For each of these and for any choice of $b_1$ we then have two possibilities for $b_0$, whereas $c_0$ is determined by the remaining variables. It follows that the above system has $p^2-p$ solutions.

For $ {\mathcal A}_p := \{(a_0,a_1,b_0,b_1,c_0,c_1) \in \F_p^6 \mid M_{B^2-4AC,A} \mbox{ is irreducible in } \F_p[X] \}$ we then have
\begin{align*}
   \sharp\mathcal{A}_p &= p^2 \sum_{\substack{U\in\F_p[X],\ \deg U= 8,\\ U\text{ monic, even, irreducible }}}\mathcal N_U +O(p^5)= \frac{p^6}{8} + O(p^5)
\end{align*}
by Lemma~\ref{even_polb}. The theorem in this case then follows as before.

\section{Proof of Theorem~\ref{mainc}}\label{pmtc}
First, we observe that by the same argument given at the beginning of the previous section (which is not wasteful in this case), it suffices to prove~\eqref{tyr}. 

The case of $m=n=2$ is essentially identical to the case $*=\ell$ with the only difference of the different ordering which can be accounted for by choosing appropriately the $H_i$ in Lemma~\ref{lt}.

The case of $m=n=1$ corresponds to elliptic curves with $A=0$, $\deg(B)\leq 1$, $\deg(C-x^3)\leq1$. This is a simpler version of the case $*=0$ as now  we have $c_2=b_2=0$ (and thus $U$ has degree $\leq2$). Then,~\eqref{sys1} reduces to 
\est{
\begin{cases}
 b_1^3 &\equiv u_2,\\[0.4em]
-c_0b_1^3 + c_1b_0b_1^2  + b_0^3
&\equiv u_0.
\end{cases}
}
and one immediately sees that this system has $p^2$ solutions if $5\leq p\equiv 2\mod 3$. There are $\frac12 p^2+O(p)$ irreducible even quadratic polynomials $U$, since there are $(p-1)/2$ squares in $\mathbb{F}_p$, and so $M_{B,C}$ is irreducible mod~$ p$ for $p^4/2+O(p^3)$ choices of the parameters modulo $p$. The result then follows by Lemma~\ref{lt}.

The case of $m=1,n=2$ corresponds to elliptic curves with $\deg(A)=0$, $\deg(B)\leq 1$, $\deg(C-x^3)\leq1$, which is very similar to the case $*=\square$, with the difference that now we have $A$ in place of $A^2$ (and thus $k$ in place of $k^2$) and $c_2=b_2=0$ (and thus $U$ has degree $\leq3$). Thus,~\eqref{sys2} becomes
\est{
\begin{cases}
 - 4k &\equiv u_3,\\
b_1^2&\equiv u_2,\\
2b_0b_1 - 4kc_1 &\equiv u_1,\\
b_0^2 - 4kc_0 &\equiv u_0.
\end{cases}
}
This system has $(1+(\frac{u_2}{p}))p$ solutions if $u_3\nequiv 0$, $O(p^2)$ solutions if $u_3\equiv0$ and $4u_0u_2\equiv u_1^2$, and $O(p)$ solutions if $u_3\equiv0$ and $4u_0u_2\nequiv u_1^2$. It follows that the number of choices of the parameters such that $B^2-4AC$ is irreducible modulo $p$ is
\begin{align*}
 \sum_{\substack{U\in\F_p[X]\text{ irred. }\\\deg U\leq 4}}\bigg(1+\bigg(\frac{u_3}{p}\bigg)\bigg)p+O(p^5)&=\sum_{u_4\, \tn{mod } p}\ \sum_{\substack{U\in\F_p[X]\text{ irred. monic}\\\deg U= 4}}\bigg(1+\bigg(\frac{u_3 u_4}{p}\bigg)\bigg)p+O(p^5)\\
&=\frac{p^6}{4}+O(p^5)
\end{align*}
and the result follows.

Finally, the case of $m=2,n=1$ corresponds to elliptic curves with $A(X)=a X$, $a\in\Z$, $\deg(B)\leq 1$, $\deg(C-x^3)\leq1$. This is analogous to the case $*=\ell$, with the difference that now we have $a_0=0$.  Thus, for $U(X)=-4X^8+u_6X^6+\cdots +u_0\in \F_p[X]$ even and irreducible, the system~\eqref{sys3} becomes
\est{
\begin{cases}
0 & \equiv  u_6, \\
a_1 b_1^2 -   4 a_1^2 c_1& \equiv u_4, \\
2 a_1^2 b_0 b_1  -     4 a_1^3 c_0 & \equiv u_2,\\
a_1^3 b_0^2  & \equiv u_0.
\end{cases}
}
If $u_6\equiv 0$ and $5\leq p\equiv 2$ mod~$3$ this has $p^2-p$ solutions. By Lemma~\ref{even_lac_pol} we then have that $M_{B^2-4AC,A}$ is irreducible for $\frac1{8}p^5+O(p^4)$ choices of the parameters and the result follows.

\end{document}